\documentclass[12pt]{amsart}
\usepackage{amsfonts}

\newtheorem{theorem}{Theorem}

\newtheorem{definition}[theorem]{Definition}

\newtheorem{lemma}[theorem]{Lemma}

\newtheorem{proposition}[theorem]{Proposition}

\begin{document}
\title[Bernstein Operators for Extended Chebyshev Systems]{Bernstein Operators for Extended Chebyshev Systems}
\author[J.M. Aldaz, O. Kounchev and H. Render]{J.M. Aldaz, O. Kounchev and H. Render}
\thanks{The first and the last author are partially supported by Grant
MTM2006-13000-C03-03 of the D.G.I. of Spain. The last two authors
acknowledge support within the project ``Institutes Partnership'' with the
Alexander von Humboldt Foundation, Bonn.}
\address{J. M. Aldaz: PERMANENT ADDRESS: Departamento de Matem\'aticas y Computaci\'on,
Universidad  de La Rioja, 26004 Logro\~no, La Rioja, Spain.}
\email{jesus.munarrizaldaz@dmc.unirioja.es}
\address{CURRENT ADDRESS: Departamento de Matem\'aticas,
Universidad  Aut\'onoma de Madrid, Cantoblanco 28049, Madrid, Spain.}
\email{jesus.munarriz@uam.es}
\address{O. Kounchev: Institute of Mathematics and Informatics, Bulgarian Academy of
Sciences, 8 Acad. G. Bonchev Str., 1113 Sofia, Bulgaria.}
\email{kounchev@gmx.de}
\address{H. Render: Departamento de Matem\'{a}ticas y Computaci\'{o}n, Universidad de
La Rioja, Edificio Vives, Luis de Ulloa s/n., 26004 Logro\~{n}o, Espa\~{n}a.}
\email{render@gmx.de}

\begin{abstract}
Let $U_{n}\subset C^{n}\left[ a,b\right] $ be an extended Chebyshev space of
dimension $n+1$. Suppose that $f_{0}\in U_{n}$ is strictly positive and $%
f_{1}\in U_{n}$ has the property that $f_{1}/f_{0}$ is strictly increasing.
We search for conditions ensuring the existence of points $%
t_{0},...,t_{n}\in \left[ a,b\right] $ and positive coefficients $\alpha
_{0},...,\alpha _{n}$ such that for all $f\in C\left[ a,b\right]$, the
operator $B_{n}:C\left[ a,b\right] \rightarrow U_{n}$ defined by $%
B_{n}f=\sum_{k=0}^{n}f\left( t_{k}\right) \alpha _{k}p_{n,k}$ satisfies $%
B_{n}f_{0}=f_{0}$ and $B_{n}f_{1}=f_{1}.$ Here it is assumed that $%
p_{n,k},k=0,...,n$, is a Bernstein basis, defined by the property that each $%
p_{n,k}$ has a zero of order $k$ at $a$ and a zero of order $n-k$ at $b.$


\end{abstract}

\maketitle

\section{Introduction}

A recent development in CAGD is the analysis of properties of Bernstein
bases defined by trigonometric or hyperbolic polynomials, or more generally,
by elements in a Chebyshev space, see \cite{CMP04}, \cite{CMP07}, \cite
{Cost00}, \cite{CLM}, \cite{MPS97}, \cite{Mazu99}, \cite{Mazu05}, \cite
{Mazu05b}, \cite{MaPo96}, \cite{Pena02}, \cite{Pena05}, \cite{Zhan96}. The definition of a
Bernstein basis is easily stated: Let $K$ be either the field of real or
complex numbers, denoted by $\mathbb{R}$ and $\mathbb{C}$ respectively.
Assume that $U_{n}$ is a $K$-linear subspace of dimension $n+1$ of $%
C^{n}\left( I,K\right) $, the space of $n$-times continuously differentiable
$K$-valued functions on an interval $I=\left[ a,b\right] $. A system $%
p_{n,k},k=0,...,n$, in $U_{n}$ is called a \emph{Bernstein basis} for $%
\{a,b\}$ if the function $p_{n,k}$ has a zero of order $k$ at $a$ and a zero
of order $n-k$ at $b$, for $k=0,...,n$. Unlike the case of the classical
Bernstein polynomial basis on $[a,b]$, the preceding definition does not
exclude the possibility that $p_{n,k}$ might have additional zeros inside
the open interval $\left( a,b\right) $. It is easy to see that a Bernstein
basis is indeed a \emph{basis} of the linear space $U_{n}$ and that the
basis functions are unique up to a non-zero factor, see e.g. the proof of
Lemma 19 and Proposition 20 in \cite{Veli07}. The existence of a Bernstein
basis is related to the concept of an extended Chebyshev system: A $K$%
-linear subspace $U_{n}\subset C^{n}(I,K)$ of dimension $n+1$ is an \emph{%
extended Chebyshev system (or space) for the subset }$A\subset I$ if each
non-zero $f\in U_{n}$ vanishes at most $n$ times in $A$, counting
multiplicities. No condition is imposed regarding the zeros of $f$ in $%
I\setminus A$. It is not difficult to prove that a Bernstein basis exists
for $U_{n}\subset C^{n}(I,K)$ if and only if $U_{n}$ is an extended
Chebyshev system for the \emph{set} $\left\{ a,b\right\} ,$ see e.g. \cite
{CMP04}, \cite{Mazu05} for the case $K=\mathbb{R}$. In the complex case $K=%
\mathbb{C,}$ Corollary 21 in \cite{Veli07} gives the necessity; a simple
argument from linear algebra gives the existence of a function $f_{k}\in
U_{n},$ $f_{k}\neq 0,$ having at least $k$ zeros at $a$ and at least $n-k$
zeros $b,$ and using that $U_{n}$ is a Chebyshev system for $\left\{
a,b\right\} $ one concludes that $f_{k},k=0,...,n$, is indeed a Bernstein
basis for $U_{n}.$

Now assume that $U_{n}\subset C^{n}\left[ a,b\right] $ is an extended
Chebyshev system for $\{a,b\}$. The aim of this paper is to present
sufficient conditions for the existence of a Bernstein operator $B_{n}:C%
\left[ a,b\right] \rightarrow U_{n}$ generalizing the classical Bernstein
operator
\begin{equation*}
\left( B_{n}f\right) \left( x\right) =\sum_{k=0}^{n}f\left( \frac{k}{n}%
\right) \binom{n}{k}x^{k}\left( 1-x\right) ^{n-k},
\end{equation*}
in the case $U_{n}$ is the set of all polynomials of degree $\leq n$ on $%
\left[ 0,1\right] $. Here $B_{n}1=1$ and $B_{n}x=x$. For a general $U_{n}$,
we assume that a Bernstein basis $p_{n,k},$ $k=0,...,n$, exists, and we
consider linear operators $B_{n}:C\left[ a,b\right] \rightarrow U_{n}$ of
the form
\begin{equation}
B_{n}f=\sum_{k=0}^{n}f\left( t_{k}\right) \alpha _{k}p_{n,k}  \label{defB}
\end{equation}
with points $t_{0},...,t_{n}\in \left[ a,b\right] $ and positive numbers $%
\alpha _{0},...,\alpha _{n}.$ We determine the points $t_{0},...,t_{n}$ and
the coefficients $\alpha _{0},...,\alpha _{n}$ from the property that $B_{n}$
fixes two given functions $f_{0},f_{1}\in U_{n},$ i.e. from
\begin{equation}
B_{n}f_{0}=f_{0}\text{ and }B_{n}f_{1}=f_{1}.  \label{eqfix}
\end{equation}
Throughout the paper we shall assume that $f_{0}$ is strictly positive
(i.e., $f_{0}$ is real-valued and  $f_{0}  >0$)
on $%
\left[ a,b\right]$  and that
$f_{1}/f_{0}$
is strictly increasing on $[a,b]$. Condition (\ref{eqfix}) leads to
equations for the nodes $t_{0},...,t_{n}$ and the coefficients $\alpha
_{0},...,\alpha _{n}$, cf. (\ref{eqakneu}) and (\ref{eqtk}), which have at
most one solution. The main difficulty is to guarantee the solvability of
these equations, something that leads to additional assumptions.

In order to guarantee that the Bernstein operator defined in (\ref{defB}) is
positive (i.e., for all non-negative $f\in C\left[ a,b\right] $, $B_{n}f\geq
0$) we shall assume that the Bernstein basis is \emph{non-negative}, that
is, $p_{n,k}\left( x\right) \geq 0$ for all $x\in \left[ a,b\right] $ and $%
k=0,...,n.$ Proposition 4 in \cite{AKR07} shows that an extended Chebyshev
system $U_{n}$ for $\left\{ a,b\right\} $ possesses a Bernstein basis of
real-valued functions if and only if $U_{n}$ is closed under complex
conjugation, i.e. the complex conjugate function $\overline{f}$ is in $U_{n}$
for all $f\in U_{n}$. If $U_{n}$ is an
extended Chebyshev system over the \emph{interval} $\left[ a,b\right] $, closed under complex conjugation, then  the
Bernstein basis functions $p_{n,k}$ are real-valued and do not have zeros
in the open interval $\left( a,b\right) .$ Thus, under these assumptions,
 non-negative Bernstein basis exist.

In order to formulate our main result we need the following notation: For a
strictly positive function $f_{0}\in U_{n}\subset C^{n}\left[ a,b\right] $
we define the \emph{space of derivatives modulo }$f_{0}$ by
\begin{equation*}
D_{f_{0}}U_{n}:=\left\{ \frac{d}{dx}\left( \frac{f}{f_{0}}\right) :f\in
U_{n}\right\} .
\end{equation*}
Clearly $D_{f_{0}}U_{n}$ is a linear space of dimension $n.$ Next we state
our main result. It is an immediate consequence of Theorems \ref{ThmBern},
  and \ref{ThmBern2}.

\begin{theorem}
\label{ThmMain} Assume that $U_{n}$ possesses a non-negative Bernstein basis $%
p_{n,k},k=0,...,n$ for $\{a,b\}\subset I$, $a<b$. Let $f_{0}\in U_{n}$ be
strictly positive, suppose $f_{1}\in U_{n}$ has the property that $%
f_{1}/f_{0}$ is strictly increasing, and assume that $D_{f_{0}}U_{n}$
possesses a non-negative Bernstein basis $q_{n-1,k},k=0,...,n-1.$ If the
coefficients $w_{k},k=0,...,n-1,$ defined by
\begin{equation}
\frac{d}{dx}\frac{f_{1}}{f_{0}}=\sum_{k=0}^{n-1}w_{k}q_{n-1,k}  \label{condw}
\end{equation}
are non-negative, then there exist points $t_{0},...,t_{n}\in \left[ a,b%
\right] $ with $t_{0}=a$ and $t_{n}=b$ and positive coefficients $\alpha
_{0},...,\alpha _{n}$, such that the operator
\begin{equation}
B_{n}f=\sum_{k=0}^{n}f\left( t_{k}\right) \alpha _{k}p_{n,k}  \label{defBTh1}
\end{equation}
satisfies $B_{n}f_{0}=f_{0}$ and $B_{n}f_{1}=f_{1}$.
\end{theorem}

The paper is organized as follows: In the second section we shall give a
simple characterization of the existence of a Bernstein operator fixing two
functions, in terms of properties of the Bernstein basis coefficients in the
expansions of $f_{0}$ and $f_{1}$. Section 3 contains some basic results
leading to an alternative proof of the normalization property established in
\cite{CMP04} and \cite{Mazu05}. The main result in Section 4 is Theorem \ref
{ThmBern2}, from which Theorem \ref{ThmMain} follows. Section 5 is devoted
to the case where $U_{n}$ is the space of exponential polynomials; we
present improvements of some previous results from \cite{AKR07}. In the last
section we shall exhibit an example showing that the assumption $w_k \ge 0$
in (\ref{condw}) is necessary for Theorem \ref{ThmMain} to hold.

\section{On the existence of Bernstein operators}

Let us review some basic facts and notations. The $k$-th derivative of a
function $f$ is denoted by $f^{\left( k\right) }.$ A function $f\in
C^{n}\left( I,\mathbb{C}\right) $ has a \emph{zero of order }$k$ or\emph{\
of multiplicity} $k$ at a point $a\in I$ if $f\left( a\right) =...=f^{\left(
k-1\right) }\left( a\right) =0$ and $f^{\left( k\right) }\left( a\right)
\neq 0.$ We shall repeatedly use the fact that
\begin{equation}
k!\cdot \lim_{x\rightarrow a}\frac{f\left( x\right) }{\left( x-a\right) ^{k}}%
=f^{\left( k\right) }\left( a\right) .  \label{eqLim}
\end{equation}
for any function $f\in C^{\left( k\right) }(I)$ with $f\left( a\right)
=...=f^{\left( k-1\right) }\left( a\right) =0$. Of course, the same
formula holds for one side limits, which is the way we will use it.

If $p_{n,k},$ $k=0,...,n,$ is a Bernstein basis of $U_{n}$, then given any
two functions $f_{0},f_{1}\in U_{n}$ there exist coefficients $\beta
_{0},...,\beta _{n}$ and $\gamma _{0},...,\gamma _{n}$ such that
\begin{equation}
f_{0}\left( x\right) =\sum_{k=0}^{n}\beta _{k}p_{n,k}\left( x\right) \text{
and }f_{1}\left( x\right) =\sum_{k=0}^{n}\gamma _{k}p_{n,k}\left( x\right) .
\label{eqeq}
\end{equation}
Next we characterize the existence of a Bernstein operator fixing $f_{0},f_{1}\in U_{n}$, in terms of properties of the coefficients
$\beta _{0},...,\beta _{n}$ and $\gamma _{0},...,\gamma _{n}:$

\begin{theorem}
\label{ThmBern} Assume that a subspace
$U_{n}\subset C^{n}\left[ a,b\right] $
of dimension $n+1$ possesses a non-negative Bernstein basis $p_{n,k},$ $%
k=0,...,n$, for $\{a,b\}$. Suppose $f_{0},f_{1}\in U_{n}$ are such that $f_{0}>0$ and $%
f_{1}/f_{0}$ is strictly increasing on $\left[ a,b\right] $. Then there
exist (unique) points $t_{0},...,t_{n}\in \left[ a,b\right] $ and positive
coefficients $\alpha _{0},...,\alpha _{n}$ such that the operator
\begin{equation}
B_{n}f=\sum_{k=0}^{n}f\left( t_{k}\right) \alpha _{k}p_{n,k}  \label{defBTh2}
\end{equation}
fixes $f_{0}$ and $f_{1}$, if and only if the coefficients defined in (\ref
{eqeq}) satisfy $\beta _{k}>0$ and
\begin{equation}
\frac{f_{1}\left( a\right) }{f_{0}\left( a\right) }=\frac{\gamma _{0}}{\beta
_{0}}\leq \frac{\gamma _{k}}{\beta _{k}}\leq \frac{\gamma _{n}}{\beta _{n}}=%
\frac{f_{1}\left( b\right) }{f_{0}\left( b\right) }  \label{eqgamma}
\end{equation}
for all $k=0,...,n.$
\end{theorem}

\begin{proof}
Suppose there exists a Bernstein operator $B_{n}$ with the properties
described by the theorem. From $B_{n}(f_{0}) = f_{0}$ and (\ref{eqeq}) we
get
\begin{equation*}
B_{n}\left( f_{0}\right) =\sum_{k=0}^{n}f_{0}\left( t_{k}\right) \alpha
_{k}p_{n,k}=\sum_{k=0}^{n}\beta _{k}p_{n,k}.
\end{equation*}
Since $p_{n,k}$ is a basis we conclude that
\begin{equation}
f_{0}\left( t_{k}\right) \alpha _{k}=\beta _{k}.  \label{eqakneu}
\end{equation}
A similar argument, using $B_{n}f_{1}=f_{1}$, yields $f_{1}\left(
t_{k}\right) \alpha _{k}=\gamma _{k}.$ Dividing by $f_{0}\left( t_{k}\right)
\alpha _{k}=\beta _{k}$, we find that $t_{k}$ satisfies
\begin{equation}
\frac{f_{1}\left( t_{k}\right) }{f_{0}\left( t_{k}\right) }=\frac{\gamma
_{k} }{\beta _{k}}.  \label{eqtk}
\end{equation}
From $f_{0} > 0$ and $\alpha _{k} > 0$ we get $\beta _{k}> 0$. Inserting $%
x=a $ in (\ref{eqeq}) we obtain $f_{0}\left( a\right) =\beta
_{0}p_{n,0}\left( a\right) $ and $f_{1}\left( a\right) =\gamma
_{0}p_{n,0}\left( a\right)$. Thus
\begin{equation*}
\frac{f_{1}\left( a\right) }{f_{0}\left( a\right) }=\frac{\gamma _{0}}{\beta
_{0}}.
\end{equation*}
Similarly, $f_{0}\left( b\right) =\beta _{n}p_{n,n}\left( b\right) $ and $%
f_{1}\left( b\right) =\gamma _{n}p_{n,n}\left( b\right) $ imply that $\frac{%
f_{1}\left( b\right) }{f_{0}\left( b\right) }=\frac{\gamma _{n}}{\beta _{n}}%
. $ By assumption, $f_{1}/f_{0}$ is increasing and $t_{k}\in \left[ a,b%
\right] $, so $\frac{\gamma _{0}}{\beta _{0}}\leq \frac{\gamma _{k}}{\beta
_{k}}\leq \frac{\gamma _{n}}{\beta _{n}}$ for all $k=0,...,n.$

Next, suppose that the coefficients $\beta _{0},...,\beta _{n}$ are positive
and (\ref{eqgamma}) is satisfied. Since $h:=f_{1}/f_{0}$ is strictly
increasing and continuous, the image of $h$ is just the interval $\left[
h\left( a\right) ,h\left( b\right) \right] .$ So for each $\gamma _{k}/\beta
_{k}$ there exists a unique point $t_{k}\in \left[ a,b\right] $ such that $%
h\left( t_{k}\right) =\gamma _{k}/\beta _{k}$, $k=0,...,n.$ We define $%
\alpha _{k} > 0$ by the equation $f_{0}\left( t_{k}\right) \alpha _{k}=\beta
_{k}$, using $f_{0} > 0$. Then
\begin{equation*}
B_{n}\left( f_{0}\right) =\sum_{k=0}^{n}f_{0}\left( t_{k}\right) \alpha
_{k}p_{n,k}=\sum_{k=0}^{n}\beta _{k}p_{n,k}=f_{0}.
\end{equation*}
Finally, it follows from the equations $h\left( t_{k}\right) =\gamma
_{k}/\beta _{k}$ and $f_{0}\left( t_{k}\right) \alpha _{k}=\beta _{k}$ that
\begin{equation*}
B_{n}\left( f_{1}\right) =\sum_{k=0}^{n}f_{1}\left( t_{k}\right) \alpha
_{k}p_{n,k}=\sum_{k=0}^{n}f_{0}\left( t_{k}\right) \frac{\gamma _{k}}{\beta
_{k}}\alpha _{k}p_{n,k}=\sum_{k=0}^{n}\gamma _{k}p_{n,k}=f_{1}.
\end{equation*}
\end{proof}

\section{The normalization property}

A remarkable recent result is the existence of a so-called \emph{normalized}
\emph{Bernstein basis: } Assume that $U_{n}\subset C^{n}\left( I,\mathbb{R}%
\right) $ is an extended Chebyshev system over $\left[ a,b\right] $
containing the constant function $1,$ so there exists a non-negative
Bernstein basis $p_{n,k},k=0,...,n$. In particular, for some coefficients $%
\alpha _{k},k=0,...,n,$ we have $1=\sum_{k=0}^{n}\alpha _{k}p_{n,k}.$ The
\emph{normalization property} says that the coefficients $\alpha _{k}$ are
\emph{positive,} (i.e.,, no cancellation is required in order to obtain the
constant function 1 from the basis). It is proved in \cite{CMP04} and \cite
{Mazu05} that the normalization property holds if and only if the space $%
U_{n}^{\prime }$ of all derivatives $f^{\prime }$ with $f\in U_{n}$ is an
extended Chebyshev space over $\left[ a,b\right] .$ As a byproduct of our
investigations we shall obtain an alternative proof of this fact, valid in
the more general context of subspaces of $C^{n}\left( I,K\right) $.

\begin{proposition}
\label{PropABL}Assume that $U_{n}$ has a Bernstein basis $p_{n,k},k=0,...,n.$
Let $f_{0}\in U_{n}$ be strictly positive and suppose that $D_{f_{0}}U_{n}$
has a Bernstein basis $q_{n-1,k}$, $k=0,...,n-1.$
Set $c_0 := 0$, $q_{n-1, - 1} := 0$, $d_n:= 0$, and $q_{n-1, n} := 0$. For $k=1,...,n$, define the
non-zero numbers
\begin{equation}
c_{k}:= \frac{1}{f_{0}\left( a\right) }\lim_{x\downarrow a}\frac{\frac{d}{dx}%
p_{n,k}\left( x\right) }{q_{n-1,k-1}\left(x\right)}
= \frac{1}{f_{0}\left( a\right) }\frac{p_{n,k}^{\left( k\right) }\left(
a\right) }{q_{n-1,k-1}^{\left( k-1\right) }\left( a\right) }  \label{eqPR1}
\end{equation}
and for $k=0,...,n-1$, the non-zero numbers
\begin{equation}
d_{k}:=\frac{1}{f_{0}\left( b\right) }\lim_{x\uparrow b}\frac{\frac{d}{dx}%
p_{n,k}\left( x\right) }{q_{n-1,k}\left(x\right)} = \frac{1}{%
f_{0}\left(b\right) }\frac{p_{n,k}^{\left( n-k\right) }\left( b\right) }{%
q_{n-1,k}^{\left( n-1-k\right) }\left( b\right) }.  \label{eqPR}
\end{equation}
Then for every $k=0,...,n$,
\begin{equation}
\frac{d}{dx}\frac{p_{n,k}\left( x\right) }{f_{0}\left( x\right) }
=c_{k}q_{n-1,k-1}\left( x\right) +d_{k}q_{n-1,k}\left( x\right).
\label{eqPREC}
\end{equation}
\end{proposition}

\begin{proof}
Since $(p_{n,k}/f_{0})^\prime\in D_{f_{0}}\left( U_{n}\right) ,$
there exist scalars $w_{k,i}$ such that
\begin{equation}
F_{k} (x) := \frac{d}{dx}\frac{p_{n,k}\left( x\right) }{f_{0}\left( x\right)}
= \sum_{i= 0}^{n-1} w_{k,i} q_{n-1, i}\left( x\right).
\label{eqREPR}
\end{equation}
We claim that $F_n$ has a zero of order $n-1$ at $a$.
Indeed, using (\ref{eqLim}) first with $f=p_{n,n}/f_{0}$
and then with $f=p_{n,n}$, we see that
\begin{equation}\label{j1}
F_{n}^{\left( k - 1\right) }\left( a\right) = \frac{d^{k}}{dx^{k}}\frac{p_{n,n}}{%
f_{0}} \left( a\right)  =k!\lim_{x\rightarrow a}\frac{p_{n,n}\left( x\right) }{%
\left( x-a\right) ^{k}f_{0}\left( x\right) }=\frac{k!}{f_{0}\left( a\right) }%
\lim_{x\rightarrow a}\frac{p_{n,n}\left( x\right) }{\left( x-a\right) ^{k}}=%
\frac{p_{n,n}^{\left( k\right) }\left( a\right)}{f_{0}\left( a\right) }.
\end{equation}
Since $p_{n,n}^{\left( k\right) }\left( a\right) = 0$ for
$k = 1,..., n-1$ and $p_{n,n}^{\left( k\right) }\left( a\right)
\ne 0$ when $k=n$, the claim follows.
Renaming $c_{n} := w_{n,n - 1}$, we see that (\ref{eqPREC}) holds when
$k = n$, for $q_{n-1,k}$ has exactly $k$ zeros at $a$, so $F_{n}\left( x\right)
=w_{n, n-1}q_{n-1,n-1}\left( x\right)$. Now the first equality in
(\ref{eqPR1}) (when $k = n$) follows immediately. Differentiating $n-1$ times
both sides of $F_{n}\left( x\right)
=c_{n}q_{n-1,n-1}\left( x\right)$ and using (\ref{j1}) we obtain the second equality in (\ref{eqPR1}), for $k = n$.

The corresponding results when $k = 0, \dots, n - 1$ are obtained
in an entirely analogous way; in particular, for $k = 1, 	\dots , n - 1$,
since $F_{k}$ has a zero
of order $k-1$ at $a$ and a zero of order $n-k-1$ at $b$ for $k=1,...,n-1$,
we conclude that $F_{k}=w_{k, k - 1}q_{n-1,k-1}+w_{k, k}q_{n-1,k}$.
Then we rename $c_{k} = w_{k, k - 1}$, $d_{k} = w_{k,k}$, and argue
 as before.
\end{proof}

\begin{theorem}
\label{ThmRep} Assume that $U_{n}$ has a Bernstein basis $p_{n,k},k=0,...,n.$
Let $f_{0}\in U_{n}$ be strictly positive and suppose $D_{f_{0}}U_{n}$
has a Bernstein basis $q_{n-1,k}$, $k=0,...,n-1.$ Let $c_{1},..., c_{n}$ and $%
d_{0},\dots ,d_{n-1}$ be the non-zero numbers defined in (\ref{eqPR1}) and (%
\ref{eqPR}). Then
\begin{equation}
f_{0}\left( x\right) =\frac{f_{0}\left( a\right) }{p_{n,0}\left( a\right) }%
p_{n,0}\left( x\right) +\sum_{k=1}^{n}\left( -1\right) ^{k}\frac{d_{0}\cdots
d_{k-1}}{c_{1}\cdots c_{k}}\frac{f_{0}\left( a\right) }{p_{n,0}\left(
a\right) }p_{n,k}\left( x\right) .  \label{pain2}
\end{equation}
\end{theorem}

\begin{proof}
For some coefficients $\beta _{0},...,\beta _{n}$ we have $f_{0}\left( x\right)
=\sum_{k=0}^{n}\beta _{k}p_{n,k}\left( x\right) $. Inserting $x=a$ yields $%
f_{0}\left( a\right) =\beta _{0}p_{n,0}\left( a\right) .$ Since
\begin{equation*}
1=\sum_{k=0}^{n}\beta _{k}\frac{p_{n,k}\left( x\right) }{f_{0}\left(
x\right) },
\end{equation*}
from Proposition \ref{PropABL} we obtain
\begin{equation*}
0=\frac{d}{dx}1= \sum_{k=0}^{n-1}\left( \beta _{k+1}c_{k+1}+\beta
_{k}d_{k}\right) q_{n-1,k}.
\end{equation*}
Thus,  $\beta _{k+1}=-\beta
_{k}d_{k}/c_{k+1}$ for $k=0,...,n-1,$ so $%
\beta _{k}=(-1)^{k}\beta _{0}(d_{0}\cdots d_{k-1})/(c_{1}\cdots c_{k})$ when
  $k=1,\dots ,n$, and (\ref{pain2}) follows.
\end{proof}

Next we want to discuss the normalization property, so we ask when the
coefficients defined in (\ref{pain2}) are positive.

\begin{definition}
Let $U_{n}\subset C^{n}\left[ a,b\right] $ be an extended Chebyshev system
for $\{a, b\}.$ We say that a Bernstein basis $p_{n,k}$ for $\{a, b\}$, $k=0,...,n$, is \emph{locally non-negative} at $\{a,b\}$ if there exists  a $\delta >0$ such that for all $k=0,...,n$ and
all $x\in \left[ a,a+\delta \right)
\cup \left( b-\delta ,b\right]
$ we have
$
p_{n,k}\left( x\right) \geq 0$.
\end{definition}

Let us give an example of a locally non-negative Bernstein basis for which
 non-negativity fails:

\noindent \textbf{Example 1:} Let $U_{2}$ be the linear space generated by
the functions $1,\cos x,\sin x.$ If $b\neq 2\pi k,k\in \mathbb{N},$ then $%
U_{2}$ does possess a Bernstein basis for $\left\{ 0,b\right\} ,$ given by
\begin{eqnarray*}
p_{2,2}\left( x\right) &=&1-\cos x\text{, } \\
p_{2,1}\left( x\right) &=&\sin x-\frac{\sin b}{1-\cos b}\left( 1-\cos
x\right) , \\
p_{2,0}\left( x\right) &=&1-\cos \left( x-b\right) = 1 - \cos b \cos x - \sin b \sin x.
\end{eqnarray*}
Furthermore, when $k\in \mathbb{N},k\neq 0$,  $U_{2}$ does not have a Bernstein basis for $\left\{ 0,2\pi
k\right\} $. We claim that $p_{2,k},k=0,1,2$,
is a locally non-negative Bernstein basis for $\{0, b\}$ whenever $0 < b\neq 2\pi k,k\in \mathbb{%
N}$. Indeed $p_{2,2}$ and $p_{2,0}$ are obviously non-negative.
Since
\begin{equation*}
p_{2,1}^{\prime }\left( b\right) =\cos b-\frac{\sin b}{1-\cos b}\sin b=-1<0
\end{equation*}
we see that $p_{2,1}$ is strictly decreasing in a neighborhood of $b,$ so $%
p_{2,1}\left( x\right) >p_{2,1}\left( b\right) =0$ for $x<b$ sufficiently
near to $b.$ Similarly $p_{2,1}^{\prime }\left( 0\right) >0$ implies that $%
p_{2,1}\left( x\right) >0$ for all $x>0$ sufficiently near to $0.$

Finally, for suitably chosen values of $b$ the non-negativity of
$p_{2,1}$ fails (but not for all values, consider, say, $b = \pi$).  Select,
for instance, $b\approx 3 \pi$.
Then  $\sin b \approx 0$ and $1-\cos b \approx 2$,
so $p_{2,1}$ is just $\sin x$ plus a small perturbation, and hence it changes
sign over $[0,b]$.

\begin{lemma}
\label{LemA} Let $p_{n,k}$, $k=0,...,n$, be a locally non-negative Bernstein
basis for $\left\{ a,b\right\} $. Then there exists a $\delta >0$ such
that $%
p_{n,k}^{\prime }\left( x\right) <0$ for all $x\in \left[ b-\delta ,b\right]
$ and all $k=0,...,n-1$, while
$%
p_{n,k}^{\prime }\left( x\right) > 0$ for all $x\in \left[a,  a + \delta \right]$ and all $k=1,...,n$. Thus, the numbers
$c_{k}$ defined in (\ref{eqPR1}) for $k = 1, \dots, n$ are positive, and
the numbers
$d_{k}$ defined in (\ref{eqPR}) for  $k = 0, \dots, n-1$ are negative.
\end{lemma}

\begin{proof} We prove the result about $p_{n,k}^{\prime }$ only for the right endpoint $b$,
since the
arguments for $a$ are entirely analogous.
Let us write $p_{n,k}\left( x\right) =\left( b-x\right) ^{n-k-1}g\left(
x\right) $ for $k=0,...,n-1.$ Since $p_{n,k}$ has a zero of order $n-k$
at $b$ it
is clear that $g\left( b\right) =0.$ Furthermore $g\left( x\right) \geq 0$ on $\left[ b-\tau ,b\right] $ for some $\tau >0$, since $p_{n,k}$
has this property. Then for $x\in \left( b-\tau ,b\right) $
\begin{equation*}
\frac{g\left( x\right) -g\left( b\right) }{x-b}=\frac{g\left( x\right) }{x-b}%
\leq 0.
\end{equation*}
Taking the limit as $x\rightarrow b$ we see that $g^{\prime }\left( b\right)
\leq 0.$ Since $g$ has only one zero at $b$ we conclude that $g^{\prime
}\left( b\right) <0.$ By continuity we infer that $g^{\prime }\left(
x\right) <0$ on some interval $\left( b-\varepsilon , b  \right) $.
Hence
\begin{equation*}
p_{n,k}^{\prime }\left( x\right) =-\left( n-k-1\right) \left( b-x\right)
^{n-k-2}g\left( x\right) +\left( b-x\right) ^{n-k-1}g^{\prime }\left(
x\right)
\end{equation*}
and from this we see that $p_{n,k}^{\prime }\left( x\right) <0$ on
$\left( b-\varepsilon ,b\right) .$ The preceding argument gives one $\varepsilon$ for each
$k=0,...,n-1$. To find
$\delta$, select the smallest such $\varepsilon$.

Next, recalling that
 $c_{k} =  \lim_{x\downarrow a}
p_{n,k}^\prime \left( x\right)/( f_{0}\left( a\right) q_{n-1,k - 1}\left(x\right))$, cf.  (\ref{eqPR1}), and that
$d_{k}:= \lim_{x\uparrow b}
p_{n,k}^\prime \left( x\right)/( f_{0}\left( b\right) q_{n-1,k}\left(x\right))$, cf.  (\ref{eqPR}), we obtain $c_k > 0$
and $d_k < 0$.
\end{proof}

Using Theorem \ref{ThmRep} we give a simple proof, in a more
general setting, of the normalization
property established in \cite{CMP04} and \cite{Mazu05}. Note that $f_0$ need not be constant, it is enough to
assume $f_0 > 0$ to obtain the result. Considering functions more
general than 1  will be useful
later on, when dealing with exponential polynomials.

\begin{theorem}
\label{Thm7} Assume that $U_{n}$ possesses a locally non-negative Bernstein
basis $p_{n,k},k=0,...,n$ for $a\neq b.$ Let $f_{0}\in U_{n}$ be strictly
positive and assume that $D_{f_{0}}U_{n}$ possesses a locally non-negative
Bernstein basis $q_{n-1,k},k=0,...,n-1.$ Then the coefficients $\beta
_{0},...,\beta _{n}$ in the expansion $f_{0}=\sum_{k=0}^{n}\beta _{k}p_{n,k}$
are positive.
\end{theorem}

\begin{proof}
By Theorem \ref{ThmRep} it suffices to show that for $k = 1, \dots, n$,
$c_{k} > 0$ and for $k = 0, \dots, n-1$,
$d_{k} < 0$. But  this is the content of  Lemma \ref{LemA}.
\end{proof}

\section{Proof of the main result}

We now want to prove the main result of the paper, Theorem \ref{ThmMain}, presented in the
introduction. The following theorem shows that condition (\ref{eqgamma})
in Theorem \ref{ThmBern} is satisfied, so Theorem \ref{ThmBern} implies
Theorem \ref{ThmMain}.

As in (\ref{eqeq}), we use the notation
\begin{equation*}
f_{0}\left( x\right) =\sum_{k=0}^{n}\beta _{k}p_{n,k}\left( x\right) \text{
and }f_{1}\left( x\right) =\sum_{k=0}^{n}\gamma _{k}p_{n,k}\left( x\right),
\end{equation*}
for the functions $f_{0},f_{1}\in U_{n}$ with the properties indicated
in the next theorem.

\begin{theorem}
\label{ThmBern2} Assume that $U_{n}$ possesses a locally non-negative
Bernstein basis $p_{n,k},k=0,...,n,$ for $\{a,b\}$.
Suppose $f_{0},f_{1}\in U_{n}$ are such that $f_{0}>0$ and $
f_{1}/f_{0}$ is strictly increasing on $\left[ a,b\right] $.
If $D_{f_{0}}U_{n}$ has a locally
non-negative Bernstein basis $q_{n-1,k},k=0,...,n-1$, then
\begin{equation}
\frac{f_{1}\left( a\right) }{f_{0}\left( a\right) }=\frac{\gamma _{0}}{\beta
_{0}}\leq \frac{\gamma _{k}}{\beta _{k}}\leq \frac{\gamma _{n}}{\beta _{n}}=%
\frac{f_{1}\left( b\right) }{f_{0}\left( b\right) }  \label{eqdesired}
\end{equation}
for $k=1$ and $k=n-1.$ If the coefficients $w_{k},k=0,...,n - 1,$ defined by
\begin{equation}
\frac{d}{dx}\frac{f_{1}}{f_{0}}=\sum_{k=0}^{n-1}w_{k}q_{n-1,k}  \label{eqA}
\end{equation}
are non-negative, then (\ref{eqdesired}) holds for all $k=1,...,n-1.$ If $%
w_{k}$ is positive for $k=0,...,n - 1$ then strict inequalities hold in (\ref
{eqdesired}).
\end{theorem}

\begin{proof}
Since $p_{n,k},k=0,...,n$, is a basis, there exists numbers $\delta
_{1},...,\delta _{n}$ such that
\begin{equation}
\psi _{a}:=f_{1}-\frac{f_{1}\left( a\right) }{f_{0}\left( a\right) }%
f_{0}=\sum_{k=0}^{n}\delta _{k}p_{n,k}.  \label{eqfneu}
\end{equation}
From (\ref{eqeq}) we get $\delta _{k}=\gamma _{k}-\frac{%
f_{1}\left( a\right) }{f_{0}\left( a\right) }\beta _{k}$ for $k=0,...,n.$
Since $\psi _{a}\left( a\right) =0$, we see that $\delta _{0}=0$ and therefore $%
\frac{\gamma _{0}}{\beta _{0}}=\frac{f_{1}\left( a\right) }{f_{0}\left(
a\right) }.$ If we can show that $\delta _{k}\geq 0$ for $k=1,...,n$, then we
obtain $\frac{\gamma _{0}}{\beta _{0}}\leq \frac{\gamma _{k}}{\beta _{k}%
}$. Similarly, $\delta _{k}>0$  implies that $\frac{\gamma _{0}}{\beta
_{0}}<\frac{\gamma _{k}}{\beta _{k}}$ for $k=1,...,n.$ Let us prove the
non-negativity (positivity) of $\delta _{k}$ for $k=1,...,n$ from the
corresponding assumptions for $w_{k-1}$. By writing $%
\frac{f_{1}}{f_{0}}-\frac{f_{1}\left( a\right) }{f_{0}\left( a\right) }%
=\sum_{k=1}^{n}\delta _{k}\frac{p_{n,k}}{f_{0}}$ and taking the derivative
we obtain
\begin{equation*}
\frac{d}{dx}\frac{f_{1}}{f_{0}}=\sum_{k=1}^{n}\delta _{k}\frac{d}{dx}\left(
\frac{p_{n,k}}{f_{0}}\right) .
\end{equation*}
Proposition \ref{PropABL} shows that $\frac{d}{dx}\frac{f_{1}}{f_{0}}%
=\sum_{k=1}^{n-1}\delta _{k}\left[ c_{k}q_{n-1,k-1}+d_{k}q_{n-1,k}\right]
+c_{n}\delta _{n}q_{n-1,n-1}.$ A simple calculation gives
\begin{equation*}
\frac{d}{dx}\frac{f_{1}}{f_{0}}=\delta
_{1}c_{1}q_{n-1,0}+\sum_{k=1}^{n-1}\left( \delta _{k+1}c_{k+1}+\delta
_{k}d_{k}\right) q_{n-1,k}.
\end{equation*}
 On the other hand, we have the representation (\ref{eqA}%
). Comparing coefficients yields $\delta _{1}c_{1}=w_{0}$ and
\begin{equation}
c_{k+1}\delta _{k+1}=w_{k}-\delta _{k}d_{k}  \label{eqck}
\end{equation}
for $k=1,...,n-1$. Inserting $x=a$ in (\ref{eqA}), we get
\begin{equation*}
\frac{d}{dx}\frac{f_{1}}{f_{0}}\left( a\right) =w_{0}q_{n-1,0}\left(
a\right) .
\end{equation*}
Recall that by Lemma \ref{LemA},  $c_{k} > 0$ for $k=1,...,n$ and  $%
d_{k} < 0$ for $k=0,...,n-1$.
Now $f_{1}/f_{0}$ is increasing and $q_{n-1,0}\left( a\right) >0$, so
$w_{0}\geq 0$, and therefore $\delta _{1}\geq 0$ (since $c_{1}$ is
positive). If  $w_{k}\ge 0$, it follows from (\ref{eqck}),  by induction
over $k\geq 1$,  that $\delta _{k+1}\ge 0$ (since $c_{k+1} > 0$ and
$d_{k} < 0$). Likewise, $\delta _{k+1} > 0$ when
 $w_{0},...,w_{n}$ are strictly positive.

The proof of inequality $\frac{\gamma _{k}}{\beta _{k}}\leq \frac{\gamma
_{n}}{\beta _{n}}$ follows the same path. We write
\begin{equation*}
\psi _{b}:=f_{1}-\frac{f_{1}\left( b\right) }{f_{0}\left( b\right) }%
f_{0}=\sum_{k=0}^{n}\Delta _{k}p_{n,k},
\end{equation*}
and note that it suffices to prove $\Delta _{k} < 0$ for $k=0,...,n$, since
$\gamma _{k}-\frac{f_{1}\left( b\right) }{f_{0}\left( b\right) }\beta
_{k}=\Delta _{k}.$ From $\psi _{b}\left( b\right) =0$ we infer that $\Delta
_{n}=0,$ and therefore $\frac{f_{1}\left( b\right) }{f_{0}\left( b\right) }=%
\frac{\gamma _{n}}{\beta _{n}}.$ As before, Proposition \ref
{PropABL} shows that
\begin{equation*}
\frac{d}{dx}\frac{f_{1}}{f_{0}}=\Delta
_{n-1}d_{n-1}q_{n-1,n-1}+\sum_{k=0}^{n-2}\left( \Delta _{k+1}c_{k+1}+\Delta
_{k}d_{k}\right) q_{n-1,k}.
\end{equation*}
Comparison of coefficients with those of (\ref{eqA}) gives  $w_{n-1}=\Delta _{n-1}d_{n-1}$ and the
downward defined recursion $w_{k}=\Delta _{k+1}c_{k+1}+\Delta _{k}d_{k}$ for
$k=n-2,\dots,0.$ Since
\begin{equation*}
\Delta _{n-1}d_{n-1}q_{n-1,n-1}\left( b\right) =\frac{d}{dx}\frac{f_{1}}{%
f_{0}}\left( b\right) \geq 0,
\end{equation*}
 we have $\Delta _{n-1}d_{n-1}\geq 0$, and therefore $\Delta _{n-1}\leq 0$. Now we infer inductively that $\Delta _{k}=d_{k}^{-1}\left( w_{k}-\Delta
_{k+1}c_{k+1}\right) \leq 0.$ Finally, when $w_{0},\dots,w_{n} > 0$,
the conclusiion   $\Delta _{k} < 0$ for $k=0,...,n-1$ is obtained in the
same way.
\end{proof}

\section{Applications to exponential polynomials}

Let us recall that the space of exponential polynomials for given complex
numbers $\lambda _{0},...,\lambda _{n}$ is defined by
\begin{equation*}
E_{\left( \lambda _{0},...,\lambda _{n}\right) }:=\left\{ f\in C^{\infty
}\left( \mathbb{R,C}\right) :\left( \frac{d}{dx}-\lambda _{0}\right)
\dots\left( \frac{d}{dx}-\lambda _{n}\right) f=0\right\} .
\end{equation*}
We  showed in \cite{AKR07} that a space $E_{\left( \lambda
_{0},...,\lambda _{n}\right) },$  closed under complex conjugation,
is an extended Chebyshev system over the interval $\left[ a,b\right] $
provided $%
b-a<\pi /M_{n}$, where
\begin{equation}
M_{n}:=\max \left\{ \left| \text{Im}\lambda _{j}\right| :j=0,...,n\right\}.
\label{eqMM}
\end{equation}
Under this condition, the following result is proved in \cite{AKR07} for real
numbers $\lambda _{0}<\lambda _{1}:$ There exist unique\emph{\ ordered }%
points $a=t_{0}<t_{1}<...<t_{n}=b$, and unique positive coefficients $\alpha
_{0},...,\alpha _{n}$, such that the operator $B_{n}:C\left[ a,b\right]
\rightarrow E_{\left( \lambda _{0},...,\lambda _{n}\right) }$ defined by (%
\ref{defB}) satisfies
\begin{equation}
B_{n}\left( e^{\lambda _{0}x}\right) =e^{\lambda _{0}x}\text{ and }%
B_{n}\left( e^{\lambda _{1}x}\right) =e^{\lambda _{1}x}.  \label{eqnorming}
\end{equation}
Next we show that under weaker assumptions, a slightly weaker conclusion (omitting the condition of \emph{%
ordered and distinct} points $t_{0}<t_{1}<...<t_{n})$ can be obtained
 from Theorem \ref
{ThmMain}:

\begin{theorem}
\label{ThmMain2}Let $\lambda _{0}<\lambda _{1}$ be real numbers and suppose
that $E_{\left( \lambda _{2},...,\lambda _{n}\right) }$ is an  extended
Chebyshev space for $\left[ a,b\right] $, closed under complex conjugation.
Then there exist unique points $t_{0},...,t_{n}\in \left[ a,b\right] $ and
unique positive coefficients $\alpha _{0},...,\alpha _{n}$, such that the
operator $B_{n}:C\left[ a,b\right] \rightarrow E_{\left( \lambda
_{0},...,\lambda _{n}\right) }$ defined by (\ref{defB}) fixes $e^{\lambda _{0}x}$ and $e^{\lambda _{1}x}.$
\end{theorem}

\begin{proof}
Note that $E_{\left( \lambda _{0},\lambda _{1}, \dots,\lambda _{n}\right) }$
and $E_{\left( \lambda _{1}, \dots,\lambda _{n}\right) }$ are extended
Chebyshev spaces over the interval $\left[ a,b\right] $, whenever
$\lambda _{0}$
and $\lambda _{1}$ are real and $E_{\left(
\lambda _{2},...,\lambda _{n}\right) }$ is an  extended
Chebyshev space for $\left[ a,b\right] $ closed under complex conjugation  (cf. the proof of Theorem 9 in \cite{AKR07}).
Thus $U_{n}:=E_{\left( \lambda _{0},\lambda _{1},\dots,\lambda _{n}\right) }$
possesses a non-negative Bernstein basis. We want to apply Theorem \ref
{ThmMain}: Let $f_{0}\left( x\right) =e^{\lambda _{0}x}$, let
$f_{1}\left(
x\right) =e^{\lambda _{1}x}$, and observe that $g_{0}\left( x\right)
:=f_{1}\left( x\right) /f_{0}\left( x\right) =e^{\left( \lambda _{1}-\lambda
_{0}\right) x}$ is strictly increasing, since $\lambda _{0}<\lambda _{1}$ .
Now
\begin{equation*}
\frac{d}{dx}\frac{f}{f_{0}}\left( x\right) =\frac{d}{dx}\left( f\left(
x\right) e^{-\lambda _{0}x}\right) =e^{-\lambda _{0}x}\left( \frac{d}{dx}%
-\lambda _{0}\right) f\left( x\right),
\end{equation*}
so
\begin{equation*}
D_{f_{0}}E_{\left( \lambda _{0},...,\lambda _{n}\right) }=e^{-\lambda
_{0}x}E_{\left( \lambda _{1},\dots,\lambda _{n}\right) }=E_{\left( \lambda
_{1}-\lambda _{0},\dots,\lambda _{n}-\lambda _{0}\right) }.
\end{equation*}
Applying this formula to $E_{\left( \lambda _{1}-\lambda _{0},\dots,\lambda
_{n}-\lambda _{0}\right) }$ and $g_{0}\left( x\right) =e^{\left( \lambda
_{1}-\lambda _{0}\right) x}$ we get
\begin{equation}
D_{g_{0}}E_{\left( \lambda _{1}-\lambda _{0},\dots,\lambda _{n}-\lambda
_{0}\right) }=E_{\left( \lambda _{2}-\lambda _{1},\dots,\lambda _{n}-\lambda
_{1}\right) }=e^{-\lambda _{1}x}E_{\left( \lambda _{2},...,\lambda
_{n}\right) }.  \label{eaGG}
\end{equation}
It follows that $D_{f_{0}}E_{\left( \lambda _{0},...,\lambda _{n}\right) }$
is an extended Chebyshev system over $\left[ a,b\right] ,$ so it has a
non-negative Bernstein basis $q_{n-1,k},k=0,...,n-1.$ Finally, consider the
coefficients $w_{k}$ in the representation
\begin{equation}
g_{0}\left( x\right) =e^{\left( \lambda _{1}-\lambda _{0}\right)
x}=\sum_{k=0}^{n-1}w_{k}q_{n-1,k}\left( x\right)  \label{eqgneu}
\end{equation}
Put $V_{n-1}:=D_{f_{0}}E_{\left( \lambda _{0},...,\lambda _{n}\right) }$ and
apply Theorem \ref{Thm7}: By (\ref{eaGG}) $D_{g_{0}}V_{n-1}$ is an extended
Chebyshev system over $\left[ a,b\right] ,$ so by Theorem \ref{Thm7} the
coefficients $w_k$ in (\ref{eqgneu}) are positive. Since $g_{0}^{\prime }=\left(
\lambda _{1}-\lambda _{0}\right) g_{0}$, we see that the
coefficients defined by (\ref{condw}) are positive. Thus, Theorem
\ref{ThmMain}
applies and the result follows.
\end{proof}

\begin{theorem}
Let $\lambda _{0}$ be a real number and let $\lambda _{1}=\lambda _{0}.$ Suppose
that $E_{\left( \lambda _{2},...,\lambda _{n}\right) }$ is an extended
Chebyshev space for $\left[ a,b\right] $, closed under complex conjugation.
Then there exist unique points $t_{0},...,t_{n}\in \left[ a,b\right] $, and
unique positive coefficients $\alpha _{0},...,\alpha _{n}$, such that the
operator $B_{n}:C\left[ a,b\right] \rightarrow E_{\left( \lambda
_{0},\lambda _{0},\lambda _{2}...,\lambda _{n}\right) }$ defined by (\ref
{defB}) satisfies $B_{n}e^{\lambda _{0}x}=e^{\lambda _{0}x}$ and
$B_{n}\left( xe^{\lambda _{0}x}\right) =xe^{\lambda _{0}x}.$
\end{theorem}

\begin{proof}
Set $f_{0}\left( x\right)
=e^{\lambda _{0}x}$ and $f_{1}\left( x\right) =xe^{\lambda _{0}x}.$
Then argue as in the preceding proof.
\end{proof}

It is a natural question whether one can extend Theorem \ref{ThmMain2}
to the case of two complex exponentials $e^{\lambda _{0}x}$ and $e^{\lambda
_{1}x}.$ Analyzing the first part of the proof of Theorem \ref{ThmBern}
(cf. (\ref{eqtk})) one
obtains the following necessary condition for the node $t_{k}$
\begin{equation}
e^{\left( \lambda _{1}-\lambda _{0}\right) t_{k}}=\frac{\gamma_{k}}{\beta
_{k}}  \label{eqneudeftk}
\end{equation}
where $\beta _{k}$ and $\gamma _{k}$ are defined by the equations (\ref
{eqeq}) for $k=0,...,n.$ Since $\beta _{k},\gamma _{k}$ and $\lambda
_{0},\lambda _{1}$ are complex numbers there seems to be a priori little hope to
find the possible solution $t_{k}$  in the interval $\left[ a,b%
\right] .$ However, in \cite{MoNe00} S. Morigi and M. Neamtu
are able to introduce
a Bernstein operator, based on exponential polynomials with \emph{equidistant}
exponents, which fixes two complex exponentials $e^{\lambda _{0}x}$ and $%
e^{\lambda _{1}x}$ with $\lambda _{0}=\overline{\lambda _{1}}.$ We
generalize their result to the case of not necessarily equidistant exponents $\lambda _{0},...,\lambda _{n}$. Observe that in order for
$e^{ix}, e^{-ix}$ to be an extended Chebyshev system
 over $\left[ a,b\right] $, it is necessary to have $b - a < \pi$,
as can be seen by counting the zeros of $\sin x$ (or $\cos x$)
in $[a,b]$. In general the corresponding condition on the length of
$[a,b]$ is also sufficient, cf. Theorem 9 of \cite{AKR07}.

\begin{theorem}
Let $\lambda _{0},...,\lambda _{n}$ be complex numbers such that $\lambda
_{0}$ is not real and $\lambda _{1}=\overline{\lambda _{0}}.$ Assume that
the spaces $E_{\left( \lambda _{0},...,\lambda _{n}\right) }$, $E_{\left(
\lambda _{2},...,\lambda _{n}\right) }$ and $E_{\left( \lambda _{0},\lambda
_{1}\right) }$ are extended Chebyshev systems over $\left[ a,b\right] $,
closed under complex conjugation. Then there exist unique points $%
t_{0},...,t_{n}\in \left[ a,b\right] $, and unique positive coefficients $%
\alpha _{0},...,\alpha _{n}$, such that the operator $B_{n}:C\left[ a,b%
\right] \rightarrow E_{\left( \lambda _{0},\lambda _{1},\lambda
_{2}...,\lambda _{n}\right) }$ defined by (\ref{defB}) satisfies
$B_{n}e^{\lambda _{0}x}=e^{\lambda _{0}x}$ and $B_{n}\left( e^{\lambda
_{1}x}\right) =e^{\lambda _{1}x}.$
\end{theorem}

\begin{proof}
Using \cite{AKR07}, specifically Lemma 7 and Proposition 10 (which basically say that applying a translation to
$\lambda _{0},...,\lambda _{n}$ preserves the extended Chebyshev system
property, and the new Bernstein basis is obtained from the
old one via multiplication by a suitable exponential function),
one may assume that $\lambda _{0}$ is purely imaginary.
Using the assumption that $E_{\left( \lambda _{0},\lambda
_{1}\right) }$ is an extended Chebyshev space
we may assume that $\lambda _{0}=i$, $\lambda _{1}=-i,$ and   $\left[ a,b\right] \subset \left( -\frac{1}{2}\pi ,\frac{1}{2}\pi
\right) $, by a scaling
argument (cf. \cite{AKR07}, Lemma 8, which states that multiplying each
$\lambda _{0},...,\lambda _{n}$ by a fixed $c > 0$ preserves,
 over a suitably re-scaled interval, the extended Chebyshev system
property) and a translation of $[a,b]$ using a real number (cf. \cite{AKR07}, Lemma 6, by which such translation of $[a,b]$  preserves
 the extended Chebyshev system
property).

Clearly $B_{n}$ fixes the
functions $e^{ix}$ and $e^{-ix}$ if (and only if)
$B_{n}$ fixes $f_{0}=\cos x$ and $%
f_{1}=\sin x.$ Furthermore, $U_{n}:=\left\langle e^{ix}, e^{-ix},
e^{\lambda_2 x}, \dots,  e^{\lambda_n x}\right\rangle
= \left\langle f_{0}, f_{1},
e^{\lambda_2 x}, \dots,  e^{\lambda_n x}\right\rangle$.
Now we want to apply Theorem \ref{ThmMain}. By assumption $%
U_{n} =E_{\left( \lambda _{0},...,\lambda _{n}\right) }$ is an extended
Chebyshev space over  $[a,b]$, so it has a non-negative Bernstein basis. Next we
consider the space $D_{f_{0}}U_{n}.$ For every $h\in U_{n}$ we have
\begin{equation*}
\frac{d}{dx}\frac{h}{f_{0}}=\frac{h^{\prime }\cdot f_{0}-h\cdot
f_{0}^{\prime }}{f_{0}^{2}}.
\end{equation*}
Using   $f_{0}^{\prime \prime }=-f_{0}$ we obtain
\begin{equation}
\frac{d}{dx}\left( f_{0}^{2}\frac{d}{dx}\frac{h}{f_{0}}\right) =h^{\prime
\prime }\cdot f_{0}-h\cdot f_{0}^{\prime \prime }=f_{0}\left( \frac{d^{2}}{%
dx^{2}}+1\right) h.  \label{eqff}
\end{equation}
Suppose now that $g\in D_{f_{0}}U_{n}$ had at least $n$ zeros in $\left[ a,b
\right] .$ The function $g$ is of the form
\begin{equation*}
g=\frac{d}{dx}\frac{f}{f_{0}}
\end{equation*}
for some $f\in U_{n}$, and $f_{0}^{2}g$ has at least $n$ zeros in $\left[
a,b\right] ,$ so $\left( f_{0}^{2}g\right) ^{\prime }$ has at least $n-1$
zeros in $\left[ a,b\right] $ by Rolle's theorem.
Note that $ f \in U_n $ implies that
$$
  \left( {  \frac{d^2}{ d x^2 }} + 1 \right) f
   \in E_{\left( \lambda _{2},...,\lambda _{n}\right) }.$$
Thus
 by (\ref{eqff}), with $h = f$, we have
\begin{equation*}
\left( f_{0}^{2}g\right) ^{\prime }\in f_{0}\cdot E_{\left( \lambda
_{2},...,\lambda _{n}\right) }.
\end{equation*}
Since $E_{\left( \lambda _{2},...,\lambda _{n}\right) }$ is an extended
Chebyshev space over $\left[ a,b\right] $ we conclude that $\left(
f_{0}^{2}g\right) ^{\prime }=0,$ so $f_{0}^{2}g$ is constant. But a constant with
at least $n$ zeros is identically zero, so   $g\equiv 0$, and we
have shown that $D_{f_{0}}U_{n}$ is an extended Chebyshev space over $\left[
a,b\right]$. Hence it has a non-negative Bernstein basis $%
q_{n-1,k},k=0,...,n-1.$
Let us prove now the non-negativity of the coefficients $w_k$ in (\ref{condw}). Consider
\begin{equation*}
\frac{d}{dx}\frac{f_{1}}{f_{0}}=\frac{f_{1}^{\prime }\cdot f_{0}-f_{1}\cdot
f_{0}^{\prime }}{f_{0}^{2}}=\frac{1}{f_{0}^{2}}%
=\sum_{k=0}^{n-1}w_{k}q_{n-1,k}.
\end{equation*}
We show that the coefficients in the representation
\begin{equation*}
1=\sum_{k=0}^{n-1}w_{k}f_{0}^{2}\cdot q_{n-1,k}
\end{equation*}
are positive. Since $V_{n-1}:=f_{0}^{2}\cdot D_{f_{0}}U_{n}$ obviously  has
the non-negative Bernstein basis $f_{0}^{2}\cdot q_{n-1,k},k=0,...,n-1$, and $%
g_{0}:=1$ is strictly positive, Theorem \ref{Thm7} tells us that the
coefficients $w_{k}$ are positive provided that $D_{g_{0}}V_{n-1}$ is an
extended Chebyshev space over $\left[ a,b\right] .$ Using (\ref{eqff}) we
obtain
\begin{eqnarray*}
D_{g_{0}}V_{n-1} &=&\left\{ \frac{d}{dx}h:h\in V_{n-1}\right\} =\left\{
f_{0}\left( \frac{d^{2}}{dx^{2}}+1\right) f:f\in E_{\left( i,-i,\lambda
_{2},...,\lambda _{n}\right) }\right\} \\
&=&f_{0}\cdot E_{\left( \lambda _{2},...,\lambda _{n}\right) }.
\end{eqnarray*}
Since $E_{\left( \lambda _{2},...,\lambda _{n}\right) }$ is an extended Chebyshev sytem over $[a,b]$, so is $D_{g_{0}}V_{n-1}$, and the proof is
complete.
\end{proof}

\section{A counterexample}

The following counterexample shows that in order for Theorem \ref{ThmMain}
to hold, it is not enough to assume that $f_0 > 0$, that $f_1 / f_0$ is
increasing, and that both $U_n$ and $D_{f_0} U_n$ have non-negative
Bernstein bases. Positivity of the coefficients $w_k$ in (\ref{condw}) is
essential. Letting $f_0\equiv 1$, we shall exhibit a strictly increasing
function $f_1$ on $[0, b] \subset [0, 2\pi)$ such that for all $b\ge 7\pi/4$%
, the node $t_{n-2}$ used in the definition of the Bernstein operator fixing
$f_0$ and $f_1$, must fall outside $[0,b]$.

In the literature special attention has been devoted to the linear space
generated by the functions
$
1,x,...,x^{n-1},\cos x,\sin x
$
over the interval $\left[ 0,b\right] ,b>0.$ It is well known that the linear
space $\left\langle \cos x,\sin x\right\rangle $ is an extended Chebyshev
space over $[0,b]$ whenever $b<\pi ,$ while $\left\langle 1,\cos x,\sin
x\right\rangle $ and $\left\langle 1,x,\cos x,\sin x\right\rangle $ are
extended Chebyshev spaces over $[0,b]$ if $b<2\pi .$ In \cite{CMP07} it is
shown that
\begin{equation*}
U_{4}:=\left\langle 1,x,x^{2},\cos x,\sin x\right\rangle
\end{equation*}
is an extended Chebyshev system over $[0,b]$ provided that $b<\rho \approx
8.9868189$, where $\rho $ is the first positive zero of the equation $\tan
\left( x/2\right) =x/2.$

We are going to present an explicit Bernstein basis $p_{4,k}$, $k=0,...,4$
for this specific space $U_4$ and $0 < b < 2\pi$.
Such a basis can be found, for instance, by using the algorithm from Section
3 of \cite{AKR07}.
The first basis function
\begin{equation*}
p_{4,4}\left( x\right) =\cos x-1+\frac{1}{2}x^{2} \approx \frac{x^4}{24}
\end{equation*}
is strictly positive for all real $x\neq 0$ and has a zero of order 4 at $0$%
. A basis function $p_{4,3}$ with three zeros at $0$ and (at least) one zero
at $b$ is given by
\begin{equation*}
p_{4,3}\left( x\right) =\left( \sin b-b\right) \left( \cos x-1+\frac{1}{2}%
x^{2}\right) +\left( \cos b-1+\frac{1}{2}b^{2}\right) \left( x-\sin x\right)
.
\end{equation*}
By either recalling that $U_4$ is an extended Chebyshev system over $[0,b]$,
or directly by checking that if $b<\rho $, then $p_{4,3}^{\prime }\left(
b\right) \neq 0$, we conclude that $p_{4,3}$ has exactly one zero at $b$.
Furthermore, since $p_{4,3}$ does not change sign in $(0,b)$, to check
positivity it is enough to evaluate $p_{4,3}$ at some suitably selected
point. For instance, given $b \ge 7\pi/4$, it is clear that $p_{4,3}(\pi) >
0 $.

The basis function $p_{4,2}$ is given by
\begin{eqnarray*}
p_{4,2}\left( x\right) &=&1+\frac{2\cos b-2+b\sin b}{b+b\cos b-2\sin b}x-%
\frac{1}{2b}\left( \sin b+\frac{2\cos b-2+b\sin b}{b+b\cos b-2\sin b}\left(
1-\cos b\right) \right) x^{2} \\
&&-\cos x-\frac{2\cos b-2+b\sin b}{b+b\cos b-2\sin b}\sin x.
\end{eqnarray*}
It is easy to see that $p_{4,2}$ has at least two zeros at $0$, and by
computing $p_{4,2} (b)$ and $p_{4,2}^{\prime } (b)$, that it has at least
two zeros at $b$. Hence, it has exactly two zeros at each endpoint of $[0,b]$%
, and none inside, so positivity follows by evaluating $p_{4,2}$ at some
suitably chosen point in $(0,b)$ (alternatively, we mention that if $p_{4,2}$
were not positive we could replace it with $- p_{4,2}$ and nothing in the
argument below would change). Finally, the basis functions $p_{4,0}$ and $%
p_{4,1}$ are defined by $p_{4,0}\left( x\right) :=p_{4,4}\left( b-x\right)$
and $p_{4,1}\left( x\right) :=p_{4,3}\left( b-x\right) $.

Let $f_{0}=1,$ so
\begin{equation*}
D_{f_{0}}U_{4}=\left\langle 1,x,\cos x,\sin x\right\rangle .
\end{equation*}
Thus, $U_{4}$ and $D_{f_{0}}U_{4}$ possess non-negative Bernstein bases for
every $b\in \left( 0,2\pi \right) .$ Next, consider the strictly increasing
function $f_{1}\in U_{4}$ given by
\begin{equation*}
f_{1}\left( x\right) :=1+x-\cos x.
\end{equation*}
We want to show that for certain values $b\in \left( 0,2\pi \right) $ there
does not exists a Bernstein operator (as described in Theorem \ref{ThmMain})
fixing $f_0$ and $f_1$. In order to facilitate the necessary computations we
will start with some general remarks: Assume that $p_{n,k},k=0,...,n$, is a
Bernstein basis of $U_{n}$ for $\{a,b\}$; let us compute some of the
coefficients $\beta _{0},...,\beta _{n}$ in the expression
\begin{equation*}
f=\sum_{k=0}^{n}\beta _{k}p_{n,k}.
\end{equation*}
Inserting $x=b$ yields $p_{n,n}\left( b\right) \beta _{n}=f\left( b\right) .$
Taking the derivative of $f$ at $b$ we get $f^{\prime }\left( b\right)
=\beta _{n-1}p_{n,n-1}^{\prime }\left( b\right) +\beta _{n}p_{n,n}^{\prime
}\left( b\right)$, and after multiplying by $p_{n,n}\left( b\right)$, we
obtain
\begin{equation*}
p_{n,n}\left( b\right) p_{n,n-1}^{\prime }\left( b\right) \beta
_{n-1}=f^{\prime }\left( b\right) p_{n,n}\left( b\right) -f\left( b\right)
p_{n,n}^{\prime }\left( b\right).
\end{equation*}
Consider the expression $f^{\prime \prime }\left( b\right) =\beta
_{n-2}p_{n,n-2}^{\prime \prime }\left( b\right) +\beta
_{n-1}p_{n,n-1}^{\prime \prime }\left( b\right) +\beta _{n}p_{n,n}^{\prime
\prime }\left( b\right)$. Multiplication by $p_{n,n}\left( b\right)
p_{n,n-1}^{\prime }\left( b\right)$ and a short computation shows that
\begin{eqnarray}
\beta _{n-2}p_{n,n}\left( b\right) p_{n,n-1}^{\prime }\left( b\right)
p_{n,n-2}^{\prime \prime }\left( b\right) &=&p_{n,n}\left( b\right) \left[
f^{\prime \prime }\left( b\right) p_{n,n-1}^{\prime }\left( b\right)
-f^{\prime }\left( b\right) p_{n,n-1}^{\prime \prime }\left( b\right) \right]
+  \label{eqlong} \\
&&f\left( b\right) \left[ p_{n,n}^{\prime }\left( b\right) p_{n,n-1}^{\prime
\prime }\left( b\right) -p_{n,n-1}^{\prime }\left( b\right) p_{n,n}^{\prime
\prime }\left( b\right) \right] .  \label{eqlong2}
\end{eqnarray}

\begin{proposition}
Let $n\geq 2.$ Assume that $p_{n,k},k=0,...,n$ is a locally non-negative
Bernstein basis of $U_{n}$ at $\{a,b\}$, and let $f_{0},f_{1}\in U_{n}$ be
given by $f_{0}=\sum_{k=0}^{n}\beta _{k}p_{n,k}$ and $f_{1}=\sum_{k=0}^{n}%
\gamma _{k}p_{n,.k}.$ If $f_{0}\left( b\right) >0$ and $\beta _{k}>0$ for $%
k=n,n-1,n-2$, then the requirement
\begin{equation*}
\frac{\gamma _{n-2}}{\beta _{n-2}}\leq \frac{\gamma _{n}}{\beta _{n}}
\end{equation*}
is equivalent to the inequality
\begin{equation}
\left[ f_{0}\left( b\right) f_{1}^{\prime \prime }\left( b\right)
-f_{0}^{\prime \prime }\left( b\right) f_{1}\left( b\right) \right]
p_{n,n-1}^{\prime }\left( b\right) -\left[ f_{0}\left( b\right)
f_{1}^{\prime }\left( b\right) -f_{0}^{\prime }\left( b\right) f_{1}\left(
b\right) \right] p_{n,n-1}^{\prime \prime }\left( b\right) \geq 0.
\label{eqcrit}
\end{equation}
\end{proposition}

\begin{proof}
Since $p_{n,k},k=0,...,n$ is a locally non-negative Bernstein basis at $%
\{a,b\}$, it follows from Lemma \ref{LemA} that $p_{n,n-1}^{\prime }\left(
b\right) <0,$ and $p_{n,n-2}^{\prime \prime }\left( b\right) >0.$ Replacing $%
f$ in (\ref{eqlong}) by $f_0$, we get $B_{n-2}:= \beta _{n-2}p_{n,n}\left(
b\right) p_{n,n-1}^{\prime }\left( b\right) p_{n,n-2}^{\prime \prime }\left(
b\right) < 0$, so we conclude that the condition
\begin{equation*}
\frac{G_{n-2}}{B_{n-2}}=\frac{\gamma _{n-2}}{\beta _{n-2}}\leq \frac{\gamma
_{n}}{\beta _{n}}=\frac{f_{1}\left( b\right) }{f_{0}\left( b\right) },
\end{equation*}
where $G_{n-2}:= \gamma _{n-2}p_{n,n}\left( b\right) p_{n,n-1}^{\prime
}\left( b\right) p_{n,n-2}^{\prime \prime }\left( b\right)$, is equivalent
to
\begin{equation*}
f_{0}\left( b\right) G_{n-2}-f_{1}\left( b\right) B_{n-2} \ge 0 .
\end{equation*}
But this is just (\ref{eqcrit}), as can be seen by setting $f = f_0$ and $f=
f_1$ in (\ref{eqlong}).
\end{proof}

Let us return to our example, where $f_{0}=1$, $f_{1}=1+x-\cos x$, and $n=4.$
Fix $b\in \lbrack 7\pi /4,2\pi )$, and note that $\cos b>0$, $\sin b<0$, and
$\cos b+\sin b\geq 0$. Now condition (\ref{eqcrit}) becomes
\begin{equation*}
h(b):=f_{1}^{\prime \prime }\left( b\right) p_{4,3}^{\prime }\left( b\right)
-f_{1}^{\prime }\left( b\right) p_{4,3}^{\prime \prime }\left( b\right)
\end{equation*}
\begin{eqnarray*}
&=&\left( \cos b\right) \left( -\left( \sin b-b\right) ^{2}+\left( \cos b-1+%
\frac{1}{2}b^{2}\right) \left( 1-\cos b\right) \right) \\
&&-\left( 1+\sin b\right) \left( \left( \sin b-b\right) \left( 1-\cos
b\right) +\left( \cos b-1+\frac{1}{2}b^{2}\right) \left( \sin b\right)
\right)
\end{eqnarray*}
\begin{equation*}
=\left( \cos b\sin b\right) b-\frac{1}{2}b^{2}\cos b+2\cos ^{2}b-2\cos
b+b-b\cos b-\frac{1}{2}\left( \sin b\right) b^{2}+\left( \sin b\right) b-%
\frac{1}{2}b^{2}
\end{equation*}
\begin{equation*}
<b-\frac{b^{2}}{2}<0.
\end{equation*}
Thus $\gamma _{2}/\beta _{2}>\gamma _{4}/\beta _{4}=f_{1}\left( b\right) ,$
and recalling (\ref{eqtk}), we see that the corresponding node $%
t_{2}=f_{1}^{-1}\left( \gamma _{2}/\beta _{2}\right) >b$, so it falls
outside $[0,b]$.

\end{document}